\documentclass[12pt,reqno]{amsart}
\usepackage[utf8]{inputenc}
\usepackage[margin=1.15in]{geometry}
\usepackage{hyperref}

\usepackage{amssymb}
\usepackage{tikz-cd}

\usepackage{mathrsfs,amsmath,mathtools}
\usepackage{amsfonts}
\usepackage{latexsym}

\usepackage{enumerate}
\usepackage{multicol}

\usepackage[utf8]{inputenc}

\makeindex

\newcounter{Def}[section]

\theoremstyle{definition}

\newtheorem*{rep@theorem}{\rep@title}
\newcommand{\newreptheorem}[2]{%
\newenvironment{rep#1}[1]{%
 \def\rep@title{#2 \ref{##1}}%
 \begin{rep@theorem}}%
 {\end{rep@theorem}}}
\newreptheorem{theorem}{Theorem}

\newtheorem{Definition}[Def]{Definition}
\newtheorem{Theorem}[Def]{Theorem}

\newtheorem{Remark}[Def]{Remark}
\newtheorem{Lemma}[Def]{Lemma}
\newtheorem{Proposition}[Def]{Proposition}
\newtheorem{Corollary}[Def]{Corollary}

\def\C          {\mathbb C}
\def\Z          {\mathbb Z}
\def\K          {\mathbb K}

\def\ca         {{\mathcal A}}
\def\cb         {{\mathcal B}}
\def\cc         {{\mathcal C}}
\def\cd         {{\mathcal D}}
\def\cm         {{\mathcal M}}
\def\cn         {{\mathcal N}}
\def\cl         {{\mathcal L}}
\def\cz         { {\mathcal Z}}
\def\op          {{}^{\text{op}}}
\def\ot         { {\,\otimes\,}}
\def\boxt         { {\,\boxtimes\,}}
\def\mod        {\text{Mod}}
\def\bimod        {\text{Bimod}}
\def\FP        {\text{FPdim}}
\def\ind        {\text{Ind}}
\def\res        {\text{Res}}
\def\Gsum    {\bigoplus_{g\in G} }
\newcommand{\Hom}{\operatorname{Hom}}
\newcommand{\End}{\operatorname{End}}
\newcommand{\Fun}{\operatorname{Fun}}
\newcommand{\innhom}{\underline{\Hom}}

\title{Equivariant Morita theory for graded tensor categories}

\author[C\'esar Galindo]{C\'esar Galindo}
\author[David Jaklitsch]{David Jaklitsch}
\author[Christoph Schweigert]{Christoph Schweigert}

\address{Departamento de Matem\'aticas, Universidad de los Andes, Bogot\'a, Colombia}
\email{cn.galindo1116@uniandes.edu.co}
\address{Department of Mathematics, Universit\"at Hamburg, Hamburg, Germany}
\email{david.jaklitsch@uni-hamburg.de}
\address{Department of Mathematics, Universit\"at Hamburg, Hamburg, Germany}
\email{christoph.schweigert@uni-hamburg.de}

\date{}
\begin{document}
\vspace*{-10mm}
	\begin{flushright}
		\small
		{\sffamily [ZMP-HH/21-14]} \\
		\textsf{Hamburger Beitr\"age zur 
Mathematik Nr.\ 902}
	\end{flushright}
	
\vspace{8mm}
\begin{abstract}
We extend categorical Morita equivalence to finite tensor categories graded by a finite group $G$. We show that two such categories are graded Morita equivalent if and only if their equivariant Drinfeld centers are equivalent as braided $G$-crossed tensor categories.    
\end{abstract}
\maketitle

\section{Introduction}

Higher category theory has become a dominant tool for tackling structural questions in the theory of finite tensor categories. Rather than study tensor equivalence, Morita equivalence, a broader relation, has been very useful in the classification of finite tensor categories. Morita equivalence of tensor categories can be defined basically in the same way as in ring theory; that is, two tensor categories are Morita equivalent if there exists an invertible bimodule category between them. Morita equivalence for tensor categories implies that their $2$-categories of exact module categories are $2$-equivalent. \\

In ring theory, Morita equivalent rings have isomorphic centers;
the converse is not true: there are rings with isomorphic
centers that are not Morita equivalent.
In contrast, the 2-category of (exact) representations of a 
finite tensor category is entirely determined by its Drinfeld
center. This characterization tells us that the transit from finite tensor categories to finite braided categories through the center captures the Morita equivalence relation. This fact is also relevant in the study of three-dimensional topological field theories: the Turaev-Viro construction, applied to Morita equivalent (spherical) fusion categories, gives rise to equivalent once-extended topological field theories since these theories can be obtained by applying the Reshetikhin-Turaev construction to their Drinfeld center, \cite{MR3674995} and \cite{BK1,BK2,BK3}.\\

%since these theories can be obtained from their Drinfeld center using the Reshetikhin-Turaev construction to them, \cite{MR3674995}.\\

%This fact is also relevant in the study of three-dimensional topological field theories: the Turaev-Viro construction, applied to Morita equivalent (spherical) fusion categories, gives rise to equivalent once-extended topological field theories. These theories can be obtained from their Drinfeld center only, by applying the Reshetikhin-Turaev construction to it, \cite{MR3674995}.\\

A direction that has generated intense research is the interplay of symmetry with braided tensor categories, \cite{turaev2010homotopy, ENO3, SCJZ, LKW}. In the presence of symmetries, tensor categories acquire a finer classification; in this case, the role played for braided categories and finite tensor categories is done for braided $G$-crossed categories and $G$-graded tensor categories, \cite{turaev2010homotopy}. Interestingly, the input for equivariant versions of the Turaev-Viro and Reshetikhin-Turaev TFTs are precisely (spherical) $G$-graded tensor categories and (modular) braided $G$-crossed tensor categories.
It has been shown in this case \cite[Theorem 1.1]{TUvi} 
that the Turaev-Viro construction
for a spherical $G$-graded fusion category $\cc$ and a certain $G$-braided
fusion category, the equivariant center $\cz_G(\cc)$ provide
equivalent once-extended homotopy TFTs.
\\

This paper explores equivariant Morita theory for finite tensor categories faithfully graded over a finite group $G$, its $2$-categorical interpretation, and its relation with the equivariant center introduced in \cite{GNN}, without imposing semisimplicity.
We will briefly mention some definitions which are discussed in detail in the body of the article. Let $\cc$ be a faithfully  $G$-graded finite tensor category. In Subsection \ref{subsection_2-category_of_graded_modules}, the $2$-category $\textbf{\mod}^{\text{Gr}}(\cc)$ of exact $G$-graded $\cc$-module categories is introduced and a canonical (strict) $G$-action on $\textbf{\mod}^{\text{Gr}}(\cc)$ is defined by shifting the grading. This $G$-action, although trivial to define, is crucial for our results.\\

The dual category $\cc^*_\cm\op$ of a $G$-graded $\cc$-module category has a canonical faithful $G$-grading associated with the $G$-graded structure. We say that two tensor categories $\cc$ and $\cd$ are \emph{graded Morita equivalent} if there is a $G$-graded $\cc$-module category $\cm$ such that $\cd\simeq \cc^*_\cm\op$  as  $G$-graded tensor categories. Our first main result relates the equivalence of the $2$-category with $G$-action $\textbf{\mod}^{\text{Gr}}(\cc)$ and graded Morita equivalence.

\begin{reptheorem}{graded_Morita_graded_modules}
Two $G$-graded finite tensor categories $\cc$ and $\cd$ are graded Morita equivalent if and only if $\textbf{\mod}^{\text{Gr}}(\cc)$ and $\textbf{\mod}^{\text{Gr}}(\cd)$ are equivalent as $2$-categories with $G$-action.
\end{reptheorem}

On the other hand, given a $G$-graded finite tensor category $\cc$, the equivariant center, denoted by $\cz_G(\cc)$ is a braided $G$-crossed tensor category. Our second main result is an extension of \cite[Theorem 8.12.3]{EGNO} to the graded case.

\begin{reptheorem}{charac_graded_Morita_eq}
Two $G$-graded finite tensor categories $\cc$ and $\cd$ are graded Morita equivalent if and only if $\cz_G(\cc)$ and $\cz_G(\cd)$ are equivalent as braided $G$-crossed categories.
\end{reptheorem}

Theorem \ref{charac_graded_Morita_eq} implies that, concerning Morita equivalence,  the equivariant center plays precisely the same role as the Drinfeld center in the non-graded case. We want to finish this introduction by noticing that the results of this paper open the doors to the generalization of the theory of Lagrangian algebras and Witt group in the equivariant setting. We plan to continue the study of these topics in future work. \\

The structure of the paper is as follows: In Section \ref{sec_2} some preliminary background on the theory of module categories over tensor categories
is displayed, and decomposability of module categories and exact algebras is addressed. In Section \ref{sec_3} we recall the notions and $G$-structures on categories, such as equivariantization and de-equivariantization, which are key in the equivariant setting for a finite group $G$. Then, we define graded Morita equivalence in Section \ref{sec_4} and show its $2$-categorical interpretation by relating it with the $2$-category of graded module categories. Moreover, we revisit the equivariant center and show that two graded Morita equivalent graded tensor categories have equivalent equivariant centers. Finally, in Section \ref{sec_5} we prove that the equivariant Drinfeld center completely characterizes the notion of graded Morita equivalence.

\subsection*{Acknowledgments}  DJ and CS are partially supported by the Deutsche Forschungsgemeinschaft (DFG, German Research Foundation) under Germany's Excellence Strategy - EXC 2121 - "Quantum Universe" 390833306. CG would like to thank the hospitality and excellent working conditions of the Department of Mathematics at the University of Hamburg, where he carried out this research as a Fellow of the Humboldt Foundation. CG is partially supported by the School of Science of Universidad de los Andes, Convocatoria para Finaciación de Programas de Investigación, proyecto  "Condiciones de frontera para TQFT equivariantes". We would also like to thank Viktor Ostrik for helpful correspondence.

\section{Preliminaries}\label{sec_2}
We will only consider linear abelian categories over an algebraically closed field $\K$ of characteristic zero throughout the text. A linear abelian category is finite if it is equivalent to the category of finite dimensional modules over a finite dimensional $\K$-algebra.
\subsection{Tensor categories and the dual tensor category}
We follow the standard definitions considered in \cite{EGNO}. A tensor category $\cc$ is a locally finite rigid monoidal category where the monoidal unit denoted by $\mathbf{1}$ is a simple object and the corresponding tensor product denoted by $\otimes$ is bilinear. $\cc\op$ will denote the tensor category with underlying category $\cc$, but tensor product given by $X \otimes^{\text{op}} Y:=Y \otimes X$ for $X, Y\in\cc$. Without loss of generality, we will often consider strict tensor categories to simplify computations, which is justified by MacLane's coherence theorem \cite{Mac}.\\

Given a tensor category $\cc$, a module category over $\cc$ will be denoted in general by $\cm$ and its module structure by $\otimes:\cc\times\cm\to\cm$, where we suppress coherence data obeying a pentagon axiom. In the case that $\cc$ is finite, we require that $\cm$ is finite as a linear category as well.\\

A $\cc$-module category $\cm$ is called \textit{exact} if for any projective object $P\in\cc$ and any object $M\in\cm$ then the object $P\otimes M$ is projective in $\cm$. A module category is said to be \textit{indecomposable} if it is not equivalent to a direct sum of two non-trivial module categories.\\

For $\cc$-module categories $\cm$ and $\cn$, the category of $\cc$-module functors is denoted by  $\Fun_\cc(\cm,\cn)$, and its full subcategory of right exact $\cc$-module functors by  $\Fun^r_\cc(\cm,\cn)$. If the module category $\cm$ is exact then $\Fun^r_\cc(\cm,\cn)=\Fun_\cc(\cm,\cn)$.\\ 

Given an exact indecomposable $\cc$-module category  $\cm$, the \textit{dual category of $\cc$ with respect to $\cm$} is the tensor category of $\cc$-module endofunctors of $\cm$ with product given by composition and it is denoted by $\cc_\cm^*:=\End_\cc(\cm)$.\\

\subsection{Relative Center}
Given a tensor category $\cc$ and a full tensor subcategory $\cd\subset\cc$, recall, for example from the Section 2B in \cite{GNN}, that the \textit{relative center} of $\cc$ with respect to $\cd$ is a tensor category $\cz_\cd(\cc)$ whose objects are pairs $(X,\gamma)$ where $X\in\cc$ and $\gamma$ is a half-braiding relative to $\cd$, i.e., a natural isomorphism $\gamma_{Y,X}:\,Y\otimes X\xrightarrow{\sim}X\otimes Y$ for $Y\in \cd$, obeying the corresponding hexagon axiom. The morphisms are morphisms in $\cc$ commuting with the associated half-braidings. 
Notice that the Drinfeld center $\cz(\cd)$ is a tensor subcategory of $\cz_\cd(\cc)$. The relative center $\cz_\cd(\cc)$ comes equipped with a \textit{relative braiding} with respect to $\cz(\cd)$, this is a natural isomorphism given for $(X,\gamma)\in\cz(\cd)$ and $(Y,\delta)\in\cz_\cd(\cc)$ by
$$c_{(X,\gamma),(Y,\delta)}:=\delta_{X,Y}:\;X\ot Y\xrightarrow{\;\sim\;} Y\ot X$$
which obeys the corresponding hexagon axioms.
In the case of $\cd=\cc$ we have that $\cz_\cc(\cc)=\cz(\cc)$ is the Drinfeld center of $\cc$ and the relative braiding corresponds to its braided structure.

\begin{Remark}Let $\cc$ be a finite tensor category, then the relative center with respect to a tensor subcategory $\cd$ has the following properties:
\begin{enumerate}[(i)] \label{C_exact_over_center}
    
    \item There is a commutative diagram of forgetful functors.
    \begin{equation*}
    \begin{tikzcd}
    \cz(\cc)\ar[dr,"{(X,\gamma)\mapsto X}",swap]\ar[rr,"{(X,\gamma)\mapsto (X,\gamma|_\cd)}"]&&\cz_\cd(\cc)\ar[dl,"{(X,\gamma)\mapsto X}"]\\
    &\cc&
    \end{tikzcd}
    \end{equation*}
    According to \cite[Proposition 3.39]{EO} the functor $\cz(\cc)\to\cc$ is surjective, thus the forgetful functor $F:\cz_\cd(\cc)\to\cc$ is surjective as well; this means that every $X \in \cc$ is a subquotient of an object of the form $F(Z)$ with $Z\in\cz_\cd(\cc)$.
    
    \item There is a distinguished tensor equivalence $\cz_\cd(\cc)\simeq (\cd\boxtimes\cc\op)_\cc^{*}$.\label{relative_center_dual}
    
    \item The Frobenius-Perron dimension of the relative center is given by $\FP(\cz_\cd(\cc))=\FP(\cd)\,\FP(\cc)$.
    
    \item The category $\cc$ is an exact module category over $\cz_\cd(\cc)$, with module structure induced by the forgetful functor $F:\cz_\cd(\cc)\to\cc$: This follows from (\ref{relative_center_dual}) and \cite[Lemma 7.12.7]{EGNO}.
\end{enumerate}
\end{Remark}

\subsection{Algebras and module categories}
Given associative algebras $A$ and $B$ in a tensor category $\cc$, the category of right $A$-modules in $\cc$ will be denoted by $\mod_A(\cc)$. Similarly, ${}_A\mod(\cc)$ denotes the category of left $A$-modules in $\cc$ and ${}_A\bimod_B(\cc)$ the category of $(A,B)$-bimodules in $\cc$. An algebra $A$ is said to be exact (indecomposable) if the $\cc$-module category $\mod_A(\cc)$ is exact (indecomposable).\\

Module categories over a tensor category $\cc$ can be realized as categories of modules over an algebra internal to $\cc$. In order to study this connection, the notion of \textit{internal Hom} is a useful tool, as presented in \cite[Section 7.9]{EGNO}:\\

Let $\cc$ be a finite tensor category and $\cm$ a $\cc$-module category, then for every $M\in\cm$ the functor $-\ot M:\cc\to\cm$ is exact and thus it has a right adjoint $\innhom_\cm(M,-):\cm\to\cc$. This means that there is a natural isomorphism
\begin{equation}\label{inn_hom_adjunction}
    \Hom_\cm(X\ot M,N)\cong\Hom_\cc(X,\innhom_\cm(M,N))
\end{equation}
for $X\in\cc$ and $M,N\in\cm$. This definition extends to a left exact functor
\begin{equation}\label{inn_hom_functor}
    \innhom_\cm(-,-):\cm\op\times\cm\to\cc,\;(M,N)\mapsto\innhom_\cm(M,N)
\end{equation}
called the \textit{internal Hom} functor of $\cm$. We will also denote the internal Hom simply by $\innhom(-,-)$ in case it is clear which module category $\cm$ are we considering. Additionally, there are canonical natural isomorphisms
\begin{equation}\label{can_inn_hom_iso}
    \innhom(M,X\ot N)\cong X\ot\innhom(M,N),\qquad \innhom(X\ot M, N)\cong \innhom(M,N)\ot X^*
\end{equation}

Denote for $M,N\in\cm$ by
\begin{align}\label{ev_MN}
    \text{ev}_{M,N}:\innhom(M,N)\ot M\to N
\end{align}
the counit of the adjunction (\ref{inn_hom_adjunction}), and for $X\in\cc$ by $\eta_{X,M}:X\to\innhom(M,X\ot M)$ the unit of (\ref{inn_hom_adjunction}). Now for objects $M,N,L\in\cm$ define the multiplication morphism
\begin{equation}\label{inn_hom_multiplication}
    \circ_{M,N,L}:\innhom(N,L)\otimes \innhom(M,N)\to \innhom(M,L)
\end{equation} 
as the image of the following composition
\begin{equation}
    \innhom(N,L)\otimes \innhom(M,N)\otimes M\xrightarrow{\text{id}\otimes\text{ev}_{M,N}}\innhom(N,L)\otimes N
\xrightarrow{\text{ev}_{N,L}}L
\end{equation}
under the adjunction (\ref{inn_hom_adjunction}).

\begin{Remark}\label{inn_hom_algebra}
For objects $M,L\in\cm$
\begin{enumerate}[(i)]
    \item The multiplication $m:=\circ_{M,M,M}$ and unit $u_M:=\eta_{\textbf{1},M}$ provides on $\innhom(M,M)$ the structure of an associative algebra in $\cc$.
    \item Moreover, $\innhom(M,L)$ is a right $\innhom(M,M)$-module with structure morphism given by $\sigma_{M,L}:=\circ_{M,M,L}$.
    \item The functor (\ref{inn_hom_functor}) can be seen as a $\cc$-module functor
    \begin{equation}\label{modulecat_algebra}
        \cm\to\mod_A(\cc),\quad L\mapsto\left(\,\innhom(M,L),\,\sigma_{M,L}\,\right)
    \end{equation}
    where $A:=\innhom(M,M)$ is the algebra described in (i).
\end{enumerate}
\end{Remark}

\begin{Theorem}{\cite[Theorem 3.17]{EO}} \\
Let $\cc$ be a finite tensor category and $\cm$ an indecomposable exact $\cc$-module category, then (\ref{modulecat_algebra}) is an equivalence of $\cc$-module categories $\cm\simeq\mod_A(\cc)$.\label{modulecat_algebra_eq}
\end{Theorem}
Recall from \cite[Proposition 7.11.1]{EGNO} that there is an equivalence of categories
\begin{equation}\label{EW_equivalence}
    {}_A\bimod_B(\cc)\xrightarrow{\;\,\sim\;\,} \Fun^{\text{r}}_\cc(\mod_A(\cc),\mod_B(\cc)), \quad M\mapsto -\otimes_A M
\end{equation}
called \textit{the Eilenberg-Watts equivalence}. Furthermore for $A=B$ and $\cm:=\mod_A(\cc)$ this is a tensor equivalence ${}_A\bimod_A(\cc)\simeq \End^{\text{r}}_\cc(\cm)\op=\cc_\cm^*\op$, i.e., it is compatible with composition of functors and the tensor product relative to $A$.\\

In view of this, the following well-known result due to Schauenburg shows a relationship between module categories and the Drinfeld center:
\begin{Theorem}{\cite[Theorem 3.3]{Sch}}\\\label{schau_equivalence_th}
Let $A$ be an algebra in a tensor category $\cc$. There is a braided monoidal equivalence 
\begin{align}\label{schau_equivalence}
    S:\cz(\cc) \xrightarrow{\;\,\sim\;\,} \cz({}_A\bimod_A(\cc)),\qquad (X,\gamma_{\_,X})\longmapsto (A\ot X,\delta_{\_,A\ot X})
\end{align}
where for $M\in{}_A\bimod_A(\cc)$, the half-braiding $\delta_{\_,A\otimes X}$ in $\cz({}_A\bimod_A(\cc))$ is defined by the composition
$$M\otimes_AA\otimes X\cong M\otimes X\xrightarrow{\gamma_{M,X}}X\otimes M\cong X\otimes A\otimes_A M\xrightarrow{\gamma^{-1}_{A,X}\otimes_AM} A\otimes X\otimes_AM $$
\end{Theorem}

\subsection{Decomposability of module categories and exact algebras}
Given a decomposition of a module category, then the internal Hom also decomposes.
\begin{Proposition}\label{decomposition_inn_hom}
Let $\cm$ be a $\cc$-module category decomposed into $\cc$-submodule categories as $\cm=\bigoplus_{i\in I}\cm_i$. Then for $M,N\in\cm$
\begin{equation}\label{inn_hom_decomp_object}
    \innhom_\cm(M,N)=\bigoplus_{i\in I}\,\innhom_{\cm_i}(M_i,N_i)
\end{equation}
as objects in $\cc$, where $M_i$ and $N_i$ are the components of $M$ and $N$ under the decomposition of $\cm$. Decomposition (\ref{inn_hom_decomp_object}) yields to a decomposition of the algebra $\innhom_\cm(M,M)$ from Remark \ref{inn_hom_algebra}
\begin{equation}\label{inn_hom_decomp_algebra}
    \innhom_\cm(M,M)=\prod_{i\in I}\,\innhom_{\cm_i}(M_i,M_i)
\end{equation}
as algebras in $\cc$.
\begin{proof}
Given $X\in\cc$, since each $\cm_i$ is a submodule category then $\Hom_\cm(X\otimes M,\,N)=\bigoplus_{i\in I}\,\Hom_{\cm_i}(X\otimes M_i,\,N_i)$ and thus
\begin{align*}
    \Hom_\cc(X,\,\innhom_\cm(M,N))&\cong\bigoplus_{i\in I}\,\Hom_{\cc}(X,\,\innhom_{\cm_i}(M_i,N_i))\\
    &\cong\Hom_{\cc}\left(X,\;\bigoplus_{i\in I}\,\innhom_{\cm_i}(M_i,N_i)\right)
\end{align*}
it follows by the Yoneda Lemma  that (\ref{inn_hom_decomp_object}) holds. Now for every $M, N\in \cm$ we have then that the identity morphism $\text{id}_{\innhom_{\cm}(M,N)}=\bigoplus_{i\in I}\text{id}_{\innhom_{\cm_i}(M_i,N_i)}$
and thus its image under the bijection (\ref{inn_hom_adjunction}) is described by the evaluation morphisms given in (\ref{ev_MN}) for the submodule categories $\cm_i$:
\begin{align*}
    \text{ev}^\cm_{M,N}=\bigoplus_{i\in I}\,\text{ev}^{\cm_i}_{M_i,N_i} \circ p_{i}
\end{align*}
where $p_i:\innhom_\cm(M,N)\ot M\to\bigoplus_{i\in I}\,\innhom_{\cm_i}(M_i,N_i)\ot M_i$ denotes the canonical projection. It follows for the multiplication morphism (\ref{inn_hom_multiplication}) that
\begin{align*}
    \circ^{\cm}_{M,N,L}=\bigoplus_{i\in I}\;\circ^{\cm_i}_{M_i,N_i,L_i}\circ p'_{i}
\end{align*}
where $p'_{i}$ is the projection $$\innhom_{\cm}(N,L)\otimes \innhom_{\cm}(M,N)\to\bigoplus_{i\in I}\,\innhom_{\cm_i}(N_i,L_i)\otimes \innhom_{\cm_i}(M_i,N_i)$$
Similarly, for $X\in\cc$ and $M\in\cm$, the identity morphism $\text{id}_{X\otimes M}=\bigoplus_{i\in I}\,\text{id}_{X\ot M_i}$
and therefore the counit of the adjunction (\ref{inn_hom_adjunction}) is given by $\eta^\cm_{X,M}=\prod_{i\in I}\eta^{\cm_i}_{X,M_i}$ where $\prod$ denotes the morphism granted by the universal property of the product $\bigoplus_{i\in I}\,\innhom_{\cm_i}(M_i,X\otimes M_i)$. This description of the multiplication and counit morphisms of $\cm$ in terms of the multiplication and counit morphisms of the submodule categories $\cm_i$, show that the algebra $\innhom_{\cm}(M,M)$ is the product of the algebras $\innhom_{\cm_i}(M_i,M_i)$ in $\cc$.
\end{proof}
\end{Proposition}
The following statement dealing with the decomposition of exact algebras is well-known among experts, but we include a proof for completeness.
\begin{Proposition}\label{algebra_decomposition}
Let $L$ be an exact algebra in a finite tensor category $\cc$.
\begin{enumerate}[(i)]
    \item There is an algebra decomposition $L=\prod_{i\in I}\, L_i$, where $L_i$ are exact indecomposable algebras in $\cc$. 
    \item The $\cc$-module category of $L$-modules in $\cc$ can be decomposed as $$\mod_L(\cc)\simeq \bigoplus_{i\in I}\,\mod_{L_i}(\cc)$$
    \item The multi-tensor category of $L$-bimodules in $\cc$ decomposes as
    \begin{align*}
        {}_{L}\bimod_{L}(\cc)\simeq\bigoplus_{i,j\in I} {}_{L_i}\bimod_{L_j}(\cc)
    \end{align*}
\end{enumerate}
\end{Proposition}
\begin{proof}
Since $L$ is exact, according to \cite[Proposition 7.6.7]{EGNO} there is a decomposition of the form
$$\cm:=\mod_L(\cc)=\bigoplus_{i\in I}\,\cm_i$$
where $\cm_i$ is an exact indecomposable $\cc$-submodule category of $\mod_L(\cc)$ for each $i\in I$. In particular the regular module can be decomposed as $L=\bigoplus_{i\in I}\, M_i$. It follows from Remark \ref{inn_hom_algebra} and Theorem \ref{modulecat_algebra_eq} that  $L_i:=\innhom_{\cm_i}(M_i,M_i)$ is an algebra in $\cc$ and there is an equivalence of $\cc$-module categories
\begin{align*}
    \innhom_{\cm_i}(M_i,-):\,\cm_i\xrightarrow{\;\sim\;}\mod_{L_i}(\cc)
\end{align*}
for each $i\in I$, and thus the algebras $L_i$ are exact and indecomposable. These equivalences lead to an equivalence of $\cc$-module categories
\begin{align*}
    \innhom_\cm(L,-)=\bigoplus_{i\in I}\,\innhom_{\cm_i}(M_i,-):\mod_L(\cc)\xrightarrow{\;\sim\;}\bigoplus_{i\in I}\,\mod_{L_i}(\cc)
\end{align*}
which shows (ii). Now recall from \cite[Example 7.9.8]{EGNO} that the internal Hom  is given by $\innhom_\cm(M,N)=(M\otimes_L{}^*N)^*$ for $M,N\in\cm$. In particular one can verify that $\innhom_\cm(L,L)=L$ as algebras in $\cc$. It follows from Proposition \ref{decomposition_inn_hom} that $$L=\innhom_\cm(L,L)=\prod_{i\in I}\,L_i$$ as algebras in $\cc$, and thus statement (i) holds. Lastly, from the Eilenberg-Watts equivalence (\ref{EW_equivalence}) follows
\begin{align*}
        {}_{L}\bimod_{L}(\cc)\simeq \End_\cc\left(\bigoplus_{i\in I}\,\mod_{L_i}(\cc)\right)\op &\simeq \bigoplus_{i,j\in I}\,\Fun_\cc\left(\mod_{L_i}(\cc),\mod_{L_j}(\cc)\right)\\
        &\simeq \bigoplus_{i,j\in I} {}_{L_i}\bimod_{L_j}(\cc)
\end{align*}
providing the desired decomposition in (iii).
\end{proof}

\section{Equivariant setting}\label{sec_3}
This section revisits some notions and $G$-structures on categories for a finite group $G$.  The content is well-known to experts, but we would like to highlight Remark \ref{central_ext_braided_crossed}.
\subsection{Group actions on tensor categories and equivariantization}
Let $G$ be a finite group and denote  by $\underline{G}$ the strict monoidal category whose objects are elements in $G$, all morphisms are identities and the tensor product is given by the group law. Given a monoidal category $\cc$ denote by  $\text{Aut}_\otimes(\cc)$ the strict monoidal category of tensor autoequivalences of $\cc$ and morphisms monoidal natural isomorphisms. A monoidal \textit{$G$-action} on $\cc$ is the datum of a monoidal functor $T:\underline{G}\rightarrow \text{Aut}_\otimes(\cc)$. We  refer to the pair $(\cc,T)$ as a \textit{monoidal $G$-category}.\\

Let $(\cc,T)$ be a monoidal $G$-category. A \textit{$G$-equivariant object} is an object $X\in\cc$ together with a choice of isomorphisms $\{u_g:T_g(X)\xrightarrow{\sim}X\}_{g\in G}$, fulfilling a compatibility condition with the tensor structure of the $G$-action functor $T$ (see \cite[Definition 2.6]{GNN}). The category of equivariant objects is denoted by $\cc^G$ and is called the \textit{equivariantization of $(\cc,T)$}. The equivariantization $\cc^G$ inherits a monoidal structure from $\cc$.

\subsection{Graded tensor categories}
Let $G$ be a finite group and $\cc$ a tensor category. A \textit{$G$-grading} on $\cc$ consists of a decomposition 
$$\cc=\Gsum \,\cc_g$$ 
into a direct sum of full abelian subcategories, such that for $g,h\in G$ the tensor product restricts to $\otimes:\cc_g\times \cc_h\to\cc_{gh}$. If $\cc_g\neq 0$ for all $g\in G$ the $G$-grading is called \textit{faithful}. In this case we also say that $\cc$ is a \textit{$G$-extension} of the trivial component $\cc_e$ and it holds that $\FP(\cc)=|G|\,\FP(\cc_e)$.

\subsection{Braided \texorpdfstring{$G$}{G}-crossed tensor categories and central \texorpdfstring{$G$}{G}-extensions}
The following notion is an analog to the notion of braided tensor category in the equivariant setting and was introduced in \cite{Tu}.
\begin{Definition}
(Braided $G$-crossed tensor category)\\
A \textit{braided} $G$-\textit{crossed tensor category} is a tensor category $\cb$ equipped with,
\begin{itemize}
    \item A $G$-grading $\cb= \Gsum \cb_g$.
    \item A compatible monoidal $G$-action $g\mapsto T_g$, i.e., $T_g(\cb_h)\subset \cb_{ghg^{-1}}$ for all $g,h\in G$.
    \item A $G$-\textit{braiding} which consists of a natural collection of isomorphisms,
    \begin{align*}
        c_{X,Y}:X\otimes Y\xrightarrow{\,\sim\,} T_g(Y)\otimes X,\qquad X\in \cb_g, \;Y\in \cb
    \end{align*}
\end{itemize}
which satisfy certain compatibility conditions with the tensor structure $\eta_{g,h}:T_g\circ T_h\xRightarrow{\sim} T_{gh}$ of the $G$-action $T$ and the tensor structure of $T_g$. A complete definition can be found in \cite[Definition 2.10]{GNN}.
\end{Definition}
\label{braided_G_crossed}
The equivariantization $\cb^G$ of any braided $G$-crossed tensor category $\cb$, inherits the structure of a braided tensor category, as explained in \cite{Mue}:
Consider $(X,\{u_g\}_{g\in G})$ and $(Y,\{v_g\}_{g\in G})$ objects in $\cb^G$ with $X\in\cb_h$, then the braiding is given by extending additively the following composition 
\begin{align}\label{equivariant_braiding}
\Tilde{c}_{X,Y}:X\otimes Y\xrightarrow{c_{X,Y}}T_h(Y)\otimes X \xrightarrow{v_h\otimes \text{id}_{X}}Y\otimes X
\end{align}
where $c_{X,Y}$ is the $G$-braiding of $\cb$. As an additional structure $\cb^G$ contains a Tannakian subcategory $\text{Rep}(G)$, i.e.,  there is a braided fully faithful functor
\begin{align}\label{incl_equi}
    \text{Rep}(G)=\text{Vec}^G\to \cb^G
\end{align}
whose essential image is given by the possible equivariant structures on objects which are multiples of the monoidal unit of $\cb$.  

\begin{Definition}
Let $\ca$ be a braided tensor category. A \textit{central $G$-extension of $\ca$} is a $G$-extension $\cb$ of $\ca$ together with a braided tensor functor $\iota:\ca\to\cz(\cb)$ such that its composition with the forgetful functor $\cz(\cb)\to\cb$ coincides with the inclusion $\ca\to\cb$. We also just say that $\cb$ is a \textit{central $G$-extension}. 
\end{Definition}
Notice that given a central $G$-extension $\cb$ of $\ca$, the datum of the braided tensor functor $\iota:\ca\to\cz(\cb)$ is the same as a \textit{relative braiding} on $\cb$ with respect to $\ca$, i.e., a natural isomorphism
$$c_{A,B}:\,A\ot B\xrightarrow{\;\sim\;}B\ot A,\qquad \text{for}\quad A\in\ca,\;B\in\cb$$
fulfilling the hexagon axioms of a braiding.

\begin{Remark}\label{central_ext_braided_crossed}
To any braided $G$-crossed tensor category $\cb$ a central $G$-extension of $\cb_e$ is associated up to equivalence, and vice versa. This is described in \cite[Proposition 8.11]{DN} as a $2$-equivalence between the $2$-groupoids of central $G$-extensions and braided $G$-crossed tensor categories. We will therefore use both notions indistinctly.
\end{Remark}

\subsection{De-equivariantization}
The inverse construction to the equivariantization of a braided $G$-crossed tensor category is known as de-equivariantization, as formulated in detail in \cite[Section 4.4]{DGNO}.\\

Let $\cd$ be a braided tensor category together with the additional datum of a braided fully faithful functor $\text{Rep}(G)\to\cd$. The group $G$ acts by left translations on the set of functions $\text{Fun}(G,\K)$ turning it into an object in $\text{Rep}(G)$. Furthermore, it has a canonical structure of a commutative separable Frobenius algebra in $\text{Rep}(G)$. This algebra is called \textit{the regular algebra of functions}; denote by $A$ its image in $\cd$ under the functor $\text{Rep}(G)\to\cd$.
The \textit{de-equivariantization of $\cd$} is a braided $G$-crossed tensor category $\cd_G$, whose underlying category is the category of modules ${}_A\mod(\cd)$ with tensor product $\otimes_A$.\\

The processes of equivariantization and de-equivariantization are mutual inverses providing a $2$-equivalence between the $2$-categories of braided $G$-crossed fusion categories and of braided fusion categories containing $\text{Rep}(G)$. As it is mentioned at the start of \cite[Section 4.4]{DGNO} these constructions and results hold in the non-semisimple case as well, leading in particular to the following statements (for a proof without the semisimplicity assumption we also refer to \cite[Section 3.4]{J}): 

\begin{enumerate}[(i)]
    \item \label{eq_deeq_equivalence}
For every braided $G$-crossed tensor category $\cb$, there is an  equivalence $\cb\xrightarrow{\,\sim\,}(\cb^G)_G$ of braided $G$-crossed tensor categories.
    \item \label{eq_deeq_equivalence2}
Given a braided tensor category $\cd$ together with a braided fully faithful functor $\Gamma:\;\text{Rep}(G)\to\cd$, there exists a braided equivalence $\cd\xrightarrow{\,\sim\,}(\cd_G)^G$, which commutes with $\Gamma$ and the braided functor (\ref{incl_equi}) coming from the equivariantization process.

\end{enumerate}

\section{Graded Morita equivalence}\label{sec_4}
From now on we will only consider faithful gradings.
\subsection{Graded module categories}
\begin{Definition}
Let $\cc$ be a $G$-graded tensor category.
\begin{enumerate}[(i)]
    \item A \textit{$G$-graded module category over $\cc$} is a $\cc$-module category with a decomposition 
    $$\cm=\Gsum \cm_g$$ 
    into a direct sum of full subcategories, where $\cm_g\neq0$ for every $g\in G$, and such that for $g,h\in G$ the tensor action restricts to $\otimes:\cc_g\times \cm_h\to\cm_{gh}$.
    \item A $G$-graded $\cc$-module category $\cm$ is called \textit{indecomposable} if it is not equivalent to a non-trivial direct sum of  $G$-graded module categories, and is called \textit{exact} if it is exact as a $\cc$-module category.
\end{enumerate}
\end{Definition}
The following propositions regarding graded module categories seem to be well-known but are difficult to find in the literature; hence, they are spelled out.
\begin{Proposition}\label{graded_innhom}
Let $\cc$ be a $G$-graded finite tensor category and $\cm$ a $G$-graded $\cc$-module category. For $M\in \cm_g$ and $N\in \cm_h$, then $\innhom(M,N)\in\cc_{hg^{-1}}$.
\begin{proof}
It follows from the definition of the internal Hom considering that $\cc$ is graded.
\end{proof}
\end{Proposition}

\begin{Proposition}\label{modules_graded}
Let $\cc$ be a $G$-graded tensor category and let $A$ be an algebra in the trivial component $\cc_e$.
\begin{enumerate}[(i)]
    \item $\mod_A(\cc)$ is a $G$-graded $\cc$-module category with decomposition 
    $$\mod_A(\cc)=\Gsum \mod_A(\cc_g)$$
    where $\mod_A(\cc_g)$ is the subcategory of $A$-modules with underlying object in $\cc_g$.
    \item The category ${}_A\bimod_A(\cc)$ is $G$-graded with decomposition
    $${}_A\bimod_A(\cc)=\bigoplus_{g\in G}{}_A\bimod_A(\cc_g)$$
    where ${}_A\bimod_A(\cc_g)$ is the subcategory of bimodules with underlying object in $\cc_g$, and this decomposition is compatible with $\otimes_A$.
    \item Let $\cb$ be a central $G$-extension and $A$ be an exact commutative algebra in the trivial component $\cb_e$. Then $\mod_A(\cb)$ has the structure of a $G$-graded tensor category.
\end{enumerate}

\begin{proof}\hfill 

\begin{enumerate}[(i)]
\item 
Consider a collection of right $A$-modules $\{(M_g,r_g:M_g\otimes A\to M_g)\}_{g\in G}$ with $M_g\in \cc_g$, and let $M:=\Gsum M_g$, then the morphism
    \begin{align*}
    M\otimes A&=\Gsum M_g\otimes A\xrightarrow{\oplus_{g\in G}r_{g}}\Gsum M_g=M
    \end{align*}
provides a right $A$-module structure on $M$.

Conversely given a module $(M,r:M\otimes A\to M)\in \mod_A(\cc)$, since $\cc$ is graded there is a decomposition of the underlying object $M=\Gsum M_g$ in $\cc$. For each $g\in G$ one can check that the object $M_g$ acquires a right $A$-module structure via the morphism
    \begin{align*}
    r_g:M_g\otimes A\xrightarrow{\iota_g\otimes \text{id}_A}M\otimes A\xrightarrow{r}M\xrightarrow{p_g}M_g
    \end{align*}
this determines a decomposition of $(M,r)=\Gsum (M_g,r_g)$ in $\mod_A(\cc)$.

The decomposition above is compatible with the module category structure of $\mod_A(\cc)$: Given an object $X\in \cc_g$ and a module $(M,r)\in \mod_A(\cc_h)$, then $X\otimes M\in \cc_{gh}$ meaning that $(X\otimes M, \text{id}_X\otimes r)\in \mod_A(\cc_{gh})$.\\

\item The decomposition of bimodules follows in complete analogy to (i).
It remains to verify that this $G$-decomposition is compatible with $\otimes_A$:\\
Let $(M,r_M,q_M)\in{}_A\bimod_A(\cc_g)$ and $(N,r_N,q_N)\in{}_A\bimod_A(\cc_h)$, and consider the coequalizer diagram
\begin{center}
\begin{tikzcd}
 M\ot A\ot N\ar[rr,shift right=.75ex,"\text{id}\otimes q_N",swap]\ar[rr,shift left=.75ex,"r_M\otimes \text{id}"]&&M\ot N\ar[r]& M\otimes_A N
\end{tikzcd}
\end{center}
If $M\otimes_A N$ were not of degree $gh$, then the coequalizer morphism would be a zero morphism, but since it is epic, then $M\otimes_A N$ would be a zero object.

\item  Given a module $(M,r)\in\mod_A(\cb)$ the following composition defines a compatible left $A$-action on $M$
\begin{align*}
    A\ot M\xrightarrow{ c_{A,M}}M\ot A\xrightarrow{r}M
\end{align*}
where $c_{A,M}$ is the relative braiding of $\cb$. This construction induces a fully faithful functor $\mod_A(\cb)\to{}_A\bimod_A(\cb)$ and $\mod_A(\cb)$ is closed under the tensor product $\otimes_A$ in ${}_A\bimod_A(\cb)$ providing the tensor structure on $\mod_A(\cb)$.
\end{enumerate}
\end{proof}
\end{Proposition}

\subsection{Induced module categories}
Let $\cc$ be a $G$-graded finite tensor category and $\cn\simeq \mod_A(\cc_e)$ a $\cc_e$-module category. The \textit{induced module category} is defined as the $\cc$-module category $\ind_{\cc_e}^\cc(\cn):=\mod_A(\cc)$. Given a $G$-graded $\cc$-module category $\cm$, the \textit{restricted $\cc_e$-module category} is denoted by $\res_{\cc_e}^\cc(\cm):=\cm_e$.

\begin{Lemma}{\cite[Lemma 4.5]{MM}}
Let $\cc$ be a $G$-graded tensor category. 
\begin{enumerate}[(i)]
    \item $\cm$ is an exact $\cc$-module category if and only if $\res_{\cc_e}^\cc(\cm)$ is an exact $\cc_e$-module category.
    \item If $\cn$ is an exact $\cc_e$-module category, then $\ind_{\cc_e}^\cc(\cn)$ is exact as a $\cc$-module category.
\end{enumerate}
\label{exact_induction}
\end{Lemma}
There is a correspondence between exact graded module categories over a graded tensor category and exact module categories over its trivial component.
\begin{Theorem}{\cite[Theorem 3.3]{Ga}}\label{galindo}\\
Let $\cc$ be a $G$-graded finite tensor category.
Induction and restriction of module categories determine a $2$-equivalence between exact $G$-graded $\cc$-module categories and exact $\cc_e$-module categories.
\end{Theorem}
\begin{proof}
The statement in \cite[Theorem 3.3]{Ga} is for the fusion case, but the proof does not require semisimplicity, and in view of Lemma \ref{exact_induction} the result holds restricting to the class of exact module categories.
\end{proof}
The following corollary is an immediate consequence from Theorem \ref{galindo}.
\begin{Corollary}\label{exact__graded_module_algebra}
For any exact $G$-graded $\cc$-module category $\cm$, there exists an algebra $A$ in the trivial component $\cc_e$ such that $\cm\simeq \mod_A(\cc)$ as $G$-graded $\cc$-module categories.

\end{Corollary}

\subsection{Graded module functors}
\begin{Definition}
Let $\cc$ be a $G$-graded tensor category and let $\cm$ and $\cn$ be $G$-graded $\cc$-module categories.
A $\cc$-module functor $F:\cn\to\cm$ is called \textit{homogeneous of degree $g\in G$}, if $F(\cn_x)\subset \cm_{xg}$ for every $x\in G$. 
A \textit{graded $\cc$-module functor} is a module functor of trivial homogeneous degree.
\end{Definition}
The full subcategory of $\Fun_\cc(\cn,\cm)$ whose objects are homogeneous module functors of degree $g\in G$ is denoted by $\Fun_\cc(\cn,\cm)_g$.
\begin{Proposition}\label{grading_module_functors}
Let $\cm$, $\cn$ and $\cl$ be $G$-graded $\cc$-module categories, then the composition of module functors restricts to
\begin{equation}\label{comp_grading}
    \circ\;:\Fun_\cc(\cn,\cl)_g \times\Fun_\cc(\cm,\cn)_h\to\Fun_\cc(\cm,\cl)_{hg}
\end{equation}
and the category of module functors decomposes as $$\Fun_\cc(\cn,\cm)=\Gsum \Fun_\cc(\cn,\cm)_g$$
In particular, for $\cm$ exact, the dual category $\cc_\cm^{*}$ is $G\op$-graded and $\Fun_\cc(\cn,\cm)$ is a $G\op$-graded module category over $\cc_\cm^{*}$. Similarly ${\cc_\cm^{*}}\op$ is a $G$-graded tensor category and  $\Fun_\cc(\cm,\cn)$ is a $G$-graded module category over ${\cc_\cm^{*}}\op$.
\begin{proof}
It is straightforward to check the compatibility of composition with the grading as expressed in (\ref{comp_grading}).\\

Given a family of functors $\{(F_g,\phi_g)\}_{g\in G}$ with
$(F_g,\phi_g)\in \Fun_\cc(\cn,\cm)_g$, their sum is defined as the functor
\begin{align*}
    \bigoplus_{g\in G}F_g:\cn\longrightarrow\cm,\quad
    N\longmapsto\bigoplus_{g\in G}F_g(N)
\end{align*}
and the isomorphisms defined for $X\in\cc$ and $N\in\cn$ by
\begin{align*}
    (\bigoplus_{g\in G}\phi_g)_{X,N}:\bigoplus_{g\in G}F_g(X\otimes N)\xrightarrow{\oplus_{g\in G}(\phi_g)_{X,N}}\bigoplus_{g\in G}X\otimes F_g(N)=X\otimes\bigoplus_{g\in G}F_g(N)
\end{align*}
provide a $\cc$-module functor structure on $\Gsum F_g$.\\
    
Conversely for $(F,\phi)\in \Fun_\cc(\cn,\cm)$ and $g,h\in G$ consider the composition of functors
$$\cn_h\xrightarrow{\,\iota_h\,}\cn\xrightarrow{\;F\;}\cn\xrightarrow{p_{hg}}\cn_{hg}$$
and define a homogeneous functor $F_g:=\bigoplus_{h\in G}p_{hg}\circ F\circ \iota_h$ with $\cc$-module structure $(\phi_g)_{X,N}$ given for homogeneous objects $X\in\cc_x$ and $N\in\cn_h$ by
    \begin{align*}
    F_g(X\otimes N)&=F(X\otimes N)_{xhg}\xrightarrow{p_{xhg}( \phi_{X,N})} \left( X\otimes F(N)\right)_{xhg}= X\otimes \left(F(N)\right)_{hg}=X\otimes F_g(N)
    \end{align*}
then $(F_g,\phi_g)\in \Fun_\cc(\cn,\cm)_g$. Moreover, their sum correspond to a decomposition of $F$ since for $N\in\cn$
    $$\bigoplus_{g\in G}F_g(N)=\bigoplus_{h,g\in G}F(N_h)_{hg}=\bigoplus_{h\in G}F(N_h)=F(N)$$
where the last line follows since $F$ preserves finite sums.
\end{proof}
\end{Proposition}
A direct computation shows that in the equivariant setting the Eilenberg-Watts equivalence is compatible with the grading:
\begin{Proposition}\label{EW_graded}
Given algebras $A$ and $B$ in the trivial component of a $G$-graded tensor category $\cc$, the Eilenberg-Watts equivalence
$${}_A\bimod_B(\cc)\xrightarrow{\,\sim\,} \Fun^{\text{r}}_\cc(\mod_A(\cc),\mod_B(\cc))$$
given by equation (\ref{EW_equivalence}) is grading preserving.
\end{Proposition}

\subsection{Graded Morita equivalence}

\begin{Definition}\label{graded_Morita_eq}
Two $G$-graded tensor categories $\cc$ and $\cd$ are said to be \textit{graded Morita equivalent} if there exists a $G$-graded $\cc$-module category $\cm$ together with a $G$-graded tensor equivalence $\cd\simeq \cc_\cm^{*}\op$.
\end{Definition}

\begin{Remark}
In the situation of Definition \ref{graded_Morita_eq}, the module category $\cm$ is necessarily:
\begin{enumerate}[(i)]
    \item Indecomposable: Follows from the fact that the identity functor of $\cm$ has to be simple in $\cc_\cm^{*}$ (see \cite[Lemma 3.24]{EO}).
    \item Exact: Rigidity of $\cc_\cm^{*}$ implies that every endofunctor of $\cm$ is exact, which means that the $\cc$-module category $\cm$ is exact (see \cite[Proposition 3.16]{EO}).
\end{enumerate}
\end{Remark}

\begin{Remark} For a $G$-graded finite tensor category $\cc$, similarly to the non-graded case, we have the following remarks:
\begin{enumerate}[(i)]
    \item Consider the regular graded module category $\cc$, then there is a $G$-graded tensor equivalence
    $$\cc\xrightarrow{\;\sim\;}(\cc_\cc^{*})\op,\quad X\mapsto -\ot X$$
    \item Let $\cm$ be an exact $G$-graded $\cc$-module category, then $\cm$ is naturally a $G\op$-graded module category over $\cc_\cm^{*}$, and the double dual $(\cc_\cm^{*})_\cm^{*}$ is a $G$-graded tensor category.
    \item Notice that the canonical tensor equivalence is $G$-graded
    \begin{align*}
        \text{can}:\,\cc \xrightarrow{\;\sim\;} (\cc_\cm^{*})_\cm^{*}\,,\quad X\mapsto X\ot-
    \end{align*}
    taking into consideration (ii).
    \item From Corollary \ref{exact__graded_module_algebra} and Proposition \ref{EW_graded}, the notion of graded Morita equivalence between $\cc$ and $\cd$ can be described by the existence of an exact algebra $A$ in $\cc_e$ together with a $G$-graded tensor equivalence $\cd\simeq{}_A\bimod_A(\cc)$.
\end{enumerate}
\label{remark_eq_relation}
\end{Remark}

\begin{Proposition}
The notion of graded Morita equivalence is an equivalence relation on $G$-graded finite tensor categories.
\begin{proof}
Remark \ref{remark_eq_relation} (i) exhibits reflexivity as well as (iii) implies symmetry. Now transitivity follows in the same manner as in the non-graded case shown in \cite[Proposition 7.12.18]{EGNO}, if one takes into consideration that the algebras involved are always in the trivial component of the corresponding graded tensor category.
\end{proof}
\end{Proposition}

From Definition \ref{graded_Morita_eq} it immediately follows that two graded Morita equivalent graded tensor categories are also Morita equivalent just as tensor categories. As shown next, their trivial components are Morita equivalent as well.
\begin{Proposition}\label{trivial_components_Morita_equivalent}
If two $G$-graded finite tensor categories $\cc$ and $\cd$ are graded Morita equivalent, then their trivial components $\cc_e$ and $\cd_e$ are Morita equivalent.
\begin{proof}
Since $\cc$ and $\cd$ are graded Morita equivalent there is an exact $G$-graded $\cc$-module category $\cm=\mod_A(\cc)$ and we have that
\begin{align*}
    (\cc_\cm^*)_e=\End_\cc(\cm)_e\simeq{}_A\bimod_A(\cc_e)\op\simeq\End_{\cc_e}(\cm_e)=(\cc_e)_{\cm_e}^*
\end{align*}
where the first equivalence takes into account that the Eilenberg-Watts equivalence preserves the grading as shown in Proposition \ref{EW_graded}. Consequently, a graded tensor equivalence $\cd\simeq\cc_\cm^*\op$ induces a tensor equivalence $\cd_e\simeq(\cc_e)_{\cm_e}^*\op$.
\end{proof}
\end{Proposition}

\begin{Remark}\label{counter_converse}
The converse of Proposition \ref{trivial_components_Morita_equivalent} does not hold in general. Morita equivalent tensor categories can have $G$-extensions which are not graded Morita equivalent:

% If $\cc$ is a $G$-graded tensor category, any group automorphism $f\in \operatorname{Aut}(G)$, defines a new $G$-grading by $(\cc^f)_g:=\cc_{f(g)}$, where $g\in G$. Depending of the group automorphism and the graded tensor category, it is possible that $\cc$ and $\cc^f$ are not graded Morita equivalent, even though their underlying tensor categories are exactly the same. For instance, if $\cc=\text{Vec}_{G}^\omega$,  and $f\in \operatorname{Aut}(G)$ is such that $f^*(\omega)$  and $\omega$ are not cohomologous, then $\cc$ and $\cc^f$ are not graded Morita equivalent.

Consider for instance the finite cyclic group $G=\Z/p\Z$  with $p$ prime and the tensor category $\cc_e=\cd_e=\text{Vec}$. The $G$-extensions $\cc=\text{Vec}_G$ and $\cd=\text{Vec}_G^\omega$ with $\omega\in H^3(G;\C^\times)$ a non-trivial $3$-cocycle have different number of indecomposable module categories according to \cite[Example 2.1]{O}. Hence $\cc$ and $\cd$ are not Morita equivalent and therefore $\cc$ and $\cd$ cannot be $G$-graded Morita equivalent.
\end{Remark}

\subsection{The 2-category of graded module categories}\label{subsection_2-category_of_graded_modules}
For every $G$-graded tensor category $\cc$ there is associated a $2$-category $\textbf{\mod}^{\text{Gr}}(\cc)$ whose objects are exact $G$-graded $\cc$-module categories, the $1$-morphisms are graded module functors and the $2$-morphisms are module natural transformations.\\

Recall the notion of group actions on $2$-categories, as discussed for example in \cite{HSV} and \cite{BGM}. A strict $G$-action on a $2$-category $\cb$ is a collection of $2$-functors $\{g_\cdot\,:\cb\to\cb\}_{g\in G}$ such that $g_\cdot\circ h_\cdot=(gh)_\cdot$ for every $g,h\in G$. An equivalence of $2$-categories with group action is a $2$-equivalence $\Psi$ with a $G$-structure $\gamma_g:\Psi\circ g_\cdot\xRightarrow{\sim}g_\cdot\circ\Psi$ fulfilling certain conditions, see \cite[Definition 2.3]{BGM} for a complete definition.\\

The $G$-grading of $\cc$ induces an additional structure on the $2$-category $\textbf{\mod}^{\text{Gr}}(\cc)$, namely a strict left action of $G$ which is given by shifting the grading: For $\cn\in\textbf{\mod}^{\text{Gr}}(\cc)$ and $g\in G$, define a $G$-graded $\cc$-module category $g_\cdot \,\cn$ by $\cn$ as a $\cc$-module category, but with $G$-grading described by the following homogeneous components $$\left[g_\cdot \,\cn\right]_x:=\cn_{xg},\quad\text{for}\; x\in G$$
Every $1$-morphism $F:\cn\to\cl$ induces a $1$-morphism $g_\cdot \,F:g_\cdot \,\cn\to g_\cdot \,\cl,\;N\mapsto F(N)$, and the assignment on $2$-morphisms is similarly defined. \\

Notice that for $\cm,\cn\in\textbf{\mod}^{\text{Gr}}(\cc)$ and $g\in G$
$$\Fun_\cc(\cm,g_\cdot \,\cn)_e=\Fun_\cc(\cm,\cn)_g= \Fun_\cc(g^{-1}_\cdot \cm,\cn)_e$$
which are $\cc$-module functors $F:\cm\to\cn$ such that $F(\cm_x)\subset\cn_{xg}$ for all $x\in G$.\\

The action of $G$ on the graded module categories of a graded tensor category plays an important role in the notion of graded Morita equivalence.

\begin{Theorem}\label{graded_Morita_graded_modules}
Two $G$-graded finite tensor categories $\cc$ and $\cd$ are graded Morita equivalent if and only if $\textbf{\mod}^{\text{Gr}}(\cc)$ and $\textbf{\mod}^{\text{Gr}}(\cd)$ are equivalent as $2$-categories with $G$-action.

% {\bf CS: A suggestion \\
% It is necessary to keep the information encoded in the $G$-action on
% $\textbf{\mod}^{\text{Gr}}(\cc)$. Suppose that $\cc$ is $G$-graded.
% Take $\tilde\cc$ to be the same monoidal category, but with the
% grading changed by some non-trivial automorphism of $G$. Then the
% $2$-categories $\textbf{\mod}^{\text{Gr}}(\cc)$ and $\textbf{\mod}^{\text{Gr}}(\tilde\cc)$ are equivalent as $2$-categories,
% but not as $2$-categories with $G$-action.}
\begin{proof}
Given an exact graded $\cc$-module category $\cm$, from \cite[Theorem 7.12.16]{EGNO} and considering Proposition \ref{grading_module_functors} we have a $2$-equivalence 
\begin{align*}
    \Psi:\,\textbf{\mod}^{\text{Gr}}(\cc)&\longrightarrow\textbf{\mod}^{\text{Gr}}(\cc_\cm^*\op)\\
    \cn&\longmapsto \Fun_\cc(\cm,\cn)
\end{align*}
which has a strict $G$-structure, i.e., the diagram of $2$-functors
\begin{center}
    \begin{tikzcd}
    \textbf{\mod}^{\text{Gr}}(\cc)\ar[r,"\Psi"]\ar[d,"g_\cdot",swap]&\textbf{\mod}^{\text{Gr}}(\cc_\cm^*\op)\ar[d,"g_\cdot"]\\
    \textbf{\mod}^{\text{Gr}}(\cc)\ar[r,"\Psi",swap]&\textbf{\mod}^{\text{Gr}}(\cc_\cm^*\op)
    \end{tikzcd}
\end{center}
strictly commutes for every $g\in G$. 
Indeed, for $\cn\in\textbf{\mod}^{\text{Gr}}(\cc)$ notice that both $\Psi(g_\cdot \,\cn)=\Fun_\cc(\cm,g_\cdot \,\cn)$ and $g_\cdot\Psi(\cn)=g_\cdot\Fun_\cc(\cm,\cn)$ are equal to $\Fun_\cc(\cm,\cn)$ as $\cc_\cm^*\op$-module categories. Moreover, the $G$-gradings coincide:\\
For a homogeneous functor $F\in\Fun_\cc(\cm,g_\cdot \,\cn)_h$ it holds that 
$$F(\cm_x)\subset\left[g_\cdot \,\cn\right]_{xh}=\cn_{xhg}$$
for all $x\in G$, which means that $F\in \Fun_\cc(\cm,\cn)_{hg}=\left[g_\cdot\Fun_\cc(\cm,\cn)\right]_{h}$.\\

Conversely, consider a $2$-equivalence
$$\mathcal{F}:\;\textbf{\mod}^{\text{Gr}}(\cd)\simeq\textbf{\mod}^{\text{Gr}}(\cc)$$
which is compatible with the corresponding $G$-actions. Denote by $\cm:=\mathcal{F}(\cd)$ the image under $\mathcal{F}$ of the regular graded $\cd$-module category. Then for every $g\in G$ we have an equivalence
\begin{align*}
    \End_\cd(\cd)_g=\Fun_\cd(\cd,g_\cdot\,\cd)_e\simeq\Fun_\cc(\cm,\mathcal{F}(g_\cdot\,\cd))_e
    \simeq\Fun_\cc(\cm,g_\cdot\,\cm)_e=\End_\cc(\cm)_g
\end{align*}
and thus $\mathcal{F}$ induces a graded equivalence of categories
$$\Omega:\;\End_\cd(\cd)\xrightarrow{\,\sim\,}\End_\cc(\cm)=\cc_\cm^*$$
Moreover, since $\mathcal{F}$ is a $2$-functor there is a natural isomorphism $\gamma$
\begin{center}
    \begin{tikzcd}
    \End_\cd(\cd)_g\times\End_\cd(\cd)_h\ar[d,"\mathcal{F}\times\mathcal{F}",swap]\ar[rr,"\circ"]&&\End_\cd(\cd)_{hg}\ar[dll,"\gamma",Rightarrow]\ar[d,"\mathcal{F}"]\\
    \End_\cc(\cm)_g\times\End_\cc(\cm)_h\ar[rr,"\circ",swap]&&\End_\cc(\cm)_{hg}
    \end{tikzcd}
\end{center}
which endows $\Omega$ with a monoidal structure. We therefore obtain a $G$-graded tensor equivalence
$$\cd\simeq\End_\cd(\cd)\op\simeq\End_\cc(\cm)\op=\cc_\cm^*\op$$
which means that $\cc$ and $\cd$ are graded Morita equivalent.
\end{proof}
\end{Theorem}
\begin{Remark}
In Theorem \ref{graded_Morita_graded_modules} it is necessary to keep the information encoded in the $G$-action on $\textbf{\mod}^{\text{Gr}}(\cc)$. If $\cc$ is a $G$-graded tensor category, any group automorphism $f\in \operatorname{Aut}(G)$, defines a new $G$-grading by $(\cc^f)_g:=\cc_{f(g)}$, where $g\in G$. Depending of the group automorphism and the graded tensor category, it is possible that the
$2$-categories $\textbf{\mod}^{\text{Gr}}(\cc)$ and $\textbf{\mod}^{\text{Gr}}(\cc^f)$ are equivalent as $2$-categories,
but not as $2$-categories with $G$-action. For instance, if $\cc=\text{Vec}_{G}^\omega$,  and $f\in \operatorname{Aut}(G)$ is such that $f^*(\omega)$  and $\omega$ are not cohomologous, then $\cc$ and $\cc^f$ are not graded Morita equivalent, but the $2$-categories $\textbf{\mod}^{\text{Gr}}(\cc)$ and $\textbf{\mod}^{\text{Gr}}(\cc^f)$ are equivalent.
\end{Remark}

The equivariantization construction has an analog for $2$-categories with group action. Let $\cc$ be a $G$-graded finite tensor category, then for $\textbf{\mod}^{\text{Gr}}(\cc)$ the $2$-category of equivariant objects corresponds to $\textbf{\mod}(\cc)$ the $2$-category of (not necessarily graded) exact module categories over $\cc$. In order to see this consider the following:\\ 

The $2$-equivalence in Theorem \ref{galindo} is given by induction of module categories and can be described in terms of the relative Deligne product \cite{ENO3}
\begin{align}\label{induction_2equivalence}
    \textbf{\mod}(\cc_e)\to\textbf{\mod}^{\text{Gr}}(\cc),\qquad    \cn\mapsto \ind_{\cc_e}^\cc(\cn)=\cc\boxtimes_{\cc_e}\cn
\end{align}
and its inverse is given by restriction
\begin{align}\label{restriction_2equivalence}
\textbf{\mod}^{\text{Gr}}(\cc)\to\textbf{\mod}(\cc_e),\qquad   \cm\mapsto \res_{\cc_e}^\cc(\cm)=\cm_e
\end{align}
Since $G$ acts on $\textbf{\mod}^{\text{Gr}}(\cc)$ we can transport this $G$-action structure to the $2$-category $\textbf{\mod}(\cc_e)$ via the $2$-equivalence (\ref{restriction_2equivalence}). We obtain for each $g\in G$ a $2$-functor
$$g_\times\,:\textbf{\mod}(\cc_e)\to\textbf{\mod}(\cc_e),\quad \cn\mapsto
\res_{\cc_e}^\cc(g_\cdot\,\ind_{\cc_e}^\cc(\cn))=\cc_g\boxtimes_{\cc_e}\cn$$
and the $\cc_e$-bimodule equivalences $M_{g,h}:\cc_g\boxtimes_{\cc_e}\cc_h\xrightarrow{\sim}\cc_{gh}$ coming from the tensor product of $\cc$ provide the corresponding pseudonatural equivalences $g_\times\circ\, h_\times\xRightarrow{\sim}gh_\times$.\\

We notice that this $G$-action coincides with the $G$-action on $\textbf{\mod}(\cc_e)$ presented in \cite[Theorem 5.4]{BGM}. Moreover, the authors also show that the $2$-category of equivariant objects in $\textbf{\mod}(\cc_e)$ under this action corresponds to $\textbf{\mod}(\cc)$.
Consequently, by considering equivariant objects in $\textbf{\mod}^{\text{Gr}}(\cc)$, one recovers the (not necessarily graded) module categories over $\cc$.
\begin{Corollary}\label{equivariantization_module_cat}
Given a $G$-graded finite tensor category $\cc$. The equivariantization $\textbf{\mod}^{\text{Gr}}(\cc)^G$ is $2$-equivalent to $\textbf{\mod}(\cc)$.
\end{Corollary}

\subsection{The equivariant center}
Given a $G$-graded finite tensor category $\cc$, there is a construction (see \cite{GNN} or \cite{TV}) that associates to $\cc$ a braided $G$-crossed tensor category called the \textit{equivariant center} and denoted by $\cz_G(\cc)$, whose underlying tensor category is the relative center $\cz_{\cc_e}(\cc)$ of $\cc$ with respect to the trivial component $\cc_e$.\\

On the other hand, the ordinary Drinfeld center of a $G$-graded finite tensor category $\cc$ is a braided tensor category containing $\text{Rep}(G)$, i.e., there is a fully faithful braided functor
\begin{align}\label{rep_center}
    \text{Rep}(G)\to \cz(\cc),\quad (\mathbb{K}^n,\rho)\mapsto(\mathbf{1}^n,\gamma_{-,\mathbf{1}^n})
\end{align}
where for $X\in\cc_g$ the half-braiding is defined via
$\gamma_{X,\mathbf{1}^n}:X\otimes \mathbf{1}^n\xrightarrow{\text{id}_X\otimes \rho(g)} X\otimes \mathbf{1}^n=\mathbf{1}^n\otimes X$.\\
Therefore $\cz(\cc)$ is a possible input for the de-equivariantization construction. The following result shown in \cite{GNN} states that the category of equivariant objects in $\cz_G(\cc)$ is braided equivalent to the Drinfeld center $\cz(\cc)$, and consequently the equivariant center $\cz_G(\cc)$ can be described as the de-equivariantization of $\cz(\cc)$.
\begin{Theorem}{\cite[Theorem 3.5]{GNN}}
Let $\cc$ be a $G$-graded finite tensor category, there is an equivalence of braided tensor categories
\begin{align*}
    \cz_G(\cc)^G\xrightarrow{\;\sim\;} \cz(\cc)
\end{align*}
compatible with the canonical inclusions of $\text{Rep}(G)$ given by (\ref{incl_equi}) and (\ref{rep_center}).
\end{Theorem}

\begin{Proposition}\label{equivariant_centers_eq}
Let $\cc$ be a $G$-graded finite tensor category and $\cm=\mod_A(\cc)$ an exact indecomposable $G$-graded $\cc$-module category, where $A$ is an algebra in $\cc_e$. Then there is an equivalence  
\begin{align}
   \cz_G(\cc)\xrightarrow{\;\sim\;} \cz_G({}_A\bimod_A(\cc)) \xrightarrow{\;\sim\;}\cz_G(\cc_\cm^{*}\op)
\end{align}
of braided $G$-crossed tensor categories.
\begin{proof}
First, notice that the following diagram commutes:
\begin{equation}\label{S_RepG}
 \begin{tikzcd}[column sep=large]
        &&\cz(\cc)\ar[dd,"S"]\\
 \text{Rep}(G)\ar[urr,"(\ref{rep_center})"]\ar[drr,"(\ref{rep_center})",swap]&&\\
        &&\cz({}_A\bimod_A(\cc))
 \end{tikzcd}
\end{equation}
where $S$ refers to equivalence (\ref{schau_equivalence}). Explicitly we have for the $G$-graded tensor category ${}_A\bimod_A(\cc)$ that the functor (\ref{rep_center}) is given by
$$\text{Rep}(G)\to\cz({}_A\bimod_A(\cc)),
\quad (\K^n,\rho)\mapsto(A\otimes \mathbf{1}^n,\beta_{\_,A^n})$$
where for $N\in{}_A\bimod_A(\cc_g)$, the half-braiding $\beta_{N,A^n}$ is given by
$$N\otimes_AA\otimes \mathbf{1}^n\xrightarrow{\text{id}_N\otimes_A\text{id}_A\otimes\rho(g)}N\otimes_AA\otimes \mathbf{1}^n\cong N\otimes \mathbf{1}^n= \mathbf{1}^n\otimes N\cong\mathbf{1}^n\otimes A\otimes_AN=A\otimes \mathbf{1}^n\otimes_AN$$
On the other hand, composing $S$ with the functor (\ref{rep_center}) corresponding to $\cc$ leads to
$$(\mathbf{1}^n,\gamma_{-,\mathbf{1}^n})\mapsto(A\otimes \mathbf{1}^n, \delta_{\_,A^n})$$
where for $M\in {}_A\bimod_A(\cc_g)$, the half-braiding $\delta_{M,A^n}$ is given by the composition
$$M\otimes_AA\otimes \mathbf{1}^n\cong M\otimes \mathbf{1}^n\xrightarrow{\text{id}_M\otimes\rho(g)}M\otimes \mathbf{1}^n=\mathbf{1}^n\otimes M\cong \mathbf{1}^n\otimes A\otimes_A M\xrightarrow{\gamma^{-1}_{A,\mathbf{1}^n}\otimes_A\text{id}_M}A\otimes\mathbf{1}^n\otimes_A M$$
but since $A\in\cc_e$, then $\gamma_{A,\mathbf{1}^n}=\text{id}_A\otimes \rho(e)=\text{id}_{A^n}$ and thus the half-braidings $\beta$ and $\delta$ coincide, which implies the commutativity of the diagram (\ref{S_RepG}).
It follows by applying de-equivariantization, that Schauenburg's equivalence (\ref{schau_equivalence}) induces an equivalence of braided $G$-crossed tensor categories
\begin{align*}
   \cz_G(\cc)\xrightarrow{\;\sim\;} \cz_G({}_A\bimod_A(\cc))
\end{align*}
Similarly, from Proposition \ref{EW_graded} the Eilenberg-Watts equivalence is graded, but the construction of the inclusion (\ref{rep_center}) is determined by the $G$-grading, therefore the diagram
\begin{center}
 \begin{tikzcd}[column sep=large]
        &&\cz({}_A\bimod_A(\cc))\ar[dd]\\
 \text{Rep}(G)\ar[urr,"(\ref{rep_center})"]\ar[drr,"(\ref{rep_center})",swap]&&\\
        &&\cz(\cc_\cm^*\op)
 \end{tikzcd}
\end{center}
is commutative and by applying de-equivariantization, we conclude that $\cz_G({}_A\bimod_A(\cc))$ and $\cz_G(\cc_\cm^{*}\op)$ are equivalent as braided $G$-crossed tensor categories as well.
\end{proof}
\end{Proposition}

\begin{Remark}\label{S_G}
Given a $G$-graded finite tensor category $\cc$ the equivalence from Proposition \ref{equivariant_centers_eq} makes the following diagram commute
\begin{center}
    \begin{tikzcd}
    \cz(\cc)\ar[rr,"S"]\ar[d]&& \cz({}_A\bimod_A(\cc))\ar[d]\\
    \cz_G(\cc)\ar[rr,"S_G"]&& \cz_G({}_A\bimod_A(\cc))
    \end{tikzcd}
\end{center}
where the vertical arrows are the forgetful functors mentioned in Remark \ref{C_exact_over_center} (i) and $S$ is equivalence (\ref{schau_equivalence}). More explicitly $S_G$ is given by the assignment $(X,\gamma_{\_,X})\longmapsto (A\ot X,\delta_{\_,A\ot X})$,
where for $M\in{}_A\bimod_A(\cc_e)$, the half-braiding $\delta_{\_,A\otimes X}$ is defined by the composition
$$M\otimes_AA\otimes X\cong M\otimes X\xrightarrow{\gamma_{M,X}}X\otimes M\cong X\otimes A\otimes_A M\xrightarrow{\gamma^{-1}_{A,X}\otimes_AM} A\otimes X\otimes_AM$$
\end{Remark}

In particular Proposition \ref{equivariant_centers_eq} implies the following extension of Theorem \ref{schau_equivalence_th} which corresponds to the case of the trivial group:

\begin{Theorem}\label{equivalent_equivariant_centers}
If two $G$-graded finite tensor categories $\cc$ and $\cd$ are graded Morita equivalent, then their equivariant centers $\cz_G(\cc)$ and $\cz_G(\cd)$ are equivalent as braided $G$-crossed tensor categories.
\end{Theorem}

\section{Characterization of graded Morita equivalence}\label{sec_5}
Morita equivalence of finite tensor categories can be completely detected by braided equivalence of their Drinfeld centers. In the graded case an analogous result will be proven in Theorem \ref{charac_graded_Morita_eq} for the equivariant center. To this end, we need to prove the converse of Theorem \ref{equivalent_equivariant_centers} and we will closely follow the approach in \cite[Section 8.12]{EGNO} for the non-graded case.\\

Given a $G$-graded finite tensor category $\cc$, the forgetful functor
\begin{equation}\label{forgetful}
    F:\,\cz_G(\cc)\to\cc,\quad (X,\gamma)\mapsto X
\end{equation}
endows $\cc$ with the structure of an exact $G$-graded module category over $\cz_G(\cc)$, according to Remark \ref{C_exact_over_center} (iv).\\

The functor $\innhom(\textbf{1},-):\cc\to\cz_G(\cc)$ is right adjoint to $F$: From the definition of the internal Hom, given $Z\in\cz_G(\cc)$ and $X\in\cc$
\begin{align}\label{adjunction}
    \Hom_\cc(F(Z),X)=\Hom_\cc(Z\ot\textbf{1},X)\cong\Hom_{\cz_G(\cc)}(Z,\innhom(\textbf{1},X))
\end{align}
where in the first equality it has been considered that $\cz_G(\cc)$ acts on $\cc$ via $F$.\\

Notice that the adjunction (\ref{adjunction}) is a special case of (\ref{inn_hom_adjunction}) for the exact $G$-graded $\cz_G(\cc)$-module category $\cc$. Denote by $\xi_X:=\text{ev}_{\textbf{1},X}: \innhom(\textbf{1},X)\ot\textbf{1}\to X$ the counit morphism (\ref{ev_MN}). Now consider for every $X\in\cc$ the following morphisms:
$$\sigma_X:=\circ_{\textbf{1},\textbf{1},X}:\innhom(\textbf{1},X)\otimes\innhom(\textbf{1},\textbf{1})\to\innhom(\textbf{1},X)$$
i.e., the image of the composition
$$\innhom(\textbf{1},X)\ot\innhom(\textbf{1},\textbf{1})\ot\textbf{1}\xrightarrow{\text{id}\otimes\xi_{\textbf{1}}}
\innhom(\textbf{1},X)\ot\textbf{1}\xrightarrow{\xi_X} X$$
under the adjunction (\ref{adjunction}), and define
$$\rho_X:\innhom(\textbf{1},\textbf{1})\otimes\innhom(\textbf{1},X)\to\innhom(\textbf{1},X)$$
as the image of the composition
$$\innhom(\textbf{1},\textbf{1})\otimes\textbf{1}\otimes\innhom(\textbf{1},X)\xrightarrow{\xi_{\textbf{1}}\otimes \text{id}}\textbf{1}\otimes\innhom(\textbf{1},X)\cong \innhom(\textbf{1},X)\otimes\textbf{1}\xrightarrow{\xi_X} X$$
under the adjunction (\ref{adjunction}), where monoidal units should be inserted and removed where necessary.\\

From Remark \ref{inn_hom_algebra} we know that
\begin{enumerate}[(i)]
    \item $A:=(\innhom(\textbf{1},\textbf{1}),m:=\sigma_{\textbf{1}},u_\textbf{1})$ is an algebra in $\cz_G(\cc)$. Moreover, $A$ is in the trivial component $\cz_G(\cc)_e$ according to Proposition \ref{graded_innhom}.
    \item $R(X):=(\innhom(\textbf{1},X),\sigma_X)$ is a right $A$-module in $\cz_G(\cc)$, for every $X\in\cc$.
    \item  From Theorem \ref{modulecat_algebra_eq} the assignment $R:\cc\longrightarrow\mod_A(\cz_G(\cc)),\; X\mapsto R(X)$ is an equivalence of $\cz_G(\cc)$-module categories, and from Proposition \ref{graded_innhom} it is $G$-graded.
\end{enumerate}

\begin{Proposition}
\begin{enumerate}[(a)]
\item The algebra $A$ is commutative in $\cz_G(\cc)_e=\cz(\cc_e)$.
\item For every $X\in\cc$, the morphism $\rho_X:A\otimes R(X)\to R(X)$ coincides with the composition 
$$A\otimes R(X)\xrightarrow{c_{A,R(X)}}R(X)\otimes A\xrightarrow{\sigma_X}R(X)$$
providing a structure of an $A$-bimodule in $\cz_G(\cc)$ on $R(X)$.
\end{enumerate}
\begin{proof}
Given $X\in\cc$, the following diagram commutes due to the naturality of the braiding
\begin{equation}\label{diagram}
\begin{tikzcd}
\innhom(\textbf{1},\textbf{1})\ot\innhom(\textbf{1},X)\ar[dd,"c_{A,R(X)}",swap]\ar[r,"\xi_{\textbf{1}}\otimes id"]&\textbf{1}\ot\innhom(\textbf{1},X)\ar[dd,"c_{\textbf{1},R(X)}"]\ar[dr,"\sim"]&&\\
&&\innhom(\textbf{1},X)\ot\textbf{1}\ar[r,"\xi_X"]&X\\
\innhom(\textbf{1},X)\ot\innhom(\textbf{1},\textbf{1})\ar[r,"id\otimes \xi_{\textbf{1}}",swap]&\innhom(\textbf{1},X)\ot\textbf{1}\ar[ru,"=",swap]&&
\end{tikzcd}
\end{equation}

From the definition of $\sigma_X$ and $\rho_X$, the commutativity of the diagram (\ref{diagram}) translates to $\sigma_X\circ c_{A,R(X)}=\rho_X$ under the isomorphism
$$\Hom_\cc(\innhom(\textbf{1},\textbf{1})\otimes \innhom(\textbf{1},X),X)\cong\Hom_{\cz_G(\cc)}(\innhom(\textbf{1},\textbf{1})\otimes \innhom(\textbf{1},X),\innhom(\textbf{1},X))$$
coming from the adjunction (\ref{adjunction}). The case $X=\textbf{1}$ corresponds to commutativity of the algebra $A=\innhom(\textbf{1},\textbf{1})$.
\end{proof}
\end{Proposition}
In particular since $A$ is a commutative algebra in the trivial component of $\cz_G(\cc)$, it follows from Proposition \ref{modules_graded} that $\mod_A(\cz_G(\cc))$ is a $G$-graded tensor category. Now the goal is to define a tensor structure on the graded functor $R:\cc\rightarrow\mod_A(\cz_G(\cc))$. It is important to point out that the construction of such tensor structure does not involve the grading on $\cc$, and thus is reduced to the non-graded case. Define for $X,Y\in\cc$ a morphism
$$\varphi_{X,Y}:\innhom(\textbf{1},X)\otimes\innhom(\textbf{1},Y)\to\innhom(\textbf{1},X\ot Y)$$
as the image of $\xi_X\otimes\xi_Y:\innhom(\textbf{1},X)\otimes\innhom(\textbf{1},Y)\to X\ot Y$ under (\ref{adjunction}). A direct computation shows that $\varphi_{X,Y}$ is a cone under the coequalizer diagram 
\begin{equation}
\begin{tikzcd}
 R(X)\otimes A\otimes R(Y)\ar[rr,shift right=.75ex,"\text{id}\otimes \rho_Y",swap]\ar[rr,shift left=.75ex,"\sigma_X\otimes \text{id}"]&&R(X)\otimes R(Y)\ar[r]\ar[dr,"\varphi_{X,Y}",swap]& R(X)\otimes_A R(Y)\ar[d,"\widetilde{\varphi}_{X,Y}",dashed]\\
 &&&R(X\otimes Y)
\end{tikzcd}
\end{equation}
and moreover one can check that the morphisms $\widetilde{\varphi}_{X,Y}:R(X)\otimes_A R(Y)\to R(X\otimes Y)$, given by the universal property of the coequalizer, fulfill the axioms of a weak tensor structure on $R$.
Furthermore, $\widetilde{\varphi}_{X,Y}:R(X)\otimes_A R(Y)\to R(X\otimes Y)$ is an isomorphism for every $X,Y\in\cc$:
\begin{enumerate}[(i)]
    \item For an object $Z\in\cz_G(\cc)$ the canonical isomorphism (\ref{can_inn_hom_iso}) provides an isomorphism of $A$-modules $R(F(Z))=\innhom(\textbf{1},Z\ot\textbf{1})\cong Z\ot\innhom(\textbf{1},\textbf{1})=Z\ot A$, where $F$ is the forgetful functor (\ref{forgetful}) and in this case $\widetilde{\varphi}_{F(Z),Y}$ corresponds to the isomorphism
    $$ R(F(Z))\otimes_A R(Y)\cong Z\ot \innhom(\textbf{1},Y)\cong\innhom(\textbf{1},Z\ot Y)= R(F(Z)\otimes Y)$$
    where the second isomorphism comes from (\ref{can_inn_hom_iso}) once more.
    \item Every projective object $P\in\cc$ is a direct summand of an object of the form $F(Z)$: Since $F$ is surjective there are $Z\in\cz_G(\cc)$ and $W\in\cc$, such that $P$ is a subobject of $W$ and $W$ is a quotient of $F(Z)$. Now from \cite[Proposition 6.1.3]{EGNO} $P$ is injective, then $P$ is a direct summand of $W$ and therefore a quotient of $F(Z)$. But from projectivity of $P$ it follows that $P$ is a direct summand of $F(Z)$.
    \item For every projective object $P\in\cc$ the morphism $\widetilde{\varphi}_{P,Y}$ is an isomorphism: From (ii) there exists $Z\in\cz_G(\cc)$ with $F(Z)=P\oplus T$ for some $T\in\cc$ and thus $\widetilde{\varphi}_{F(Z),Y}=\widetilde{\varphi}_{P,Y}\oplus \widetilde{\varphi}_{T,Y}$. It follows that if $\widetilde{\varphi}_{P,Y}$ is not an isomorphism, then $\widetilde{\varphi}_{F(Z),Y}$ is not an isomorphism, but this is a contradiction with (i).
    \item For an arbitrary $X\in\cc$, the morphism $\widetilde{\varphi}_{X,Y}$ is an isomorphism: Consider a projective cover $p:P\to X$. By naturality of $\widetilde{\varphi}_{X,Y}$ the following diagram commutes
    \begin{center}
    \begin{tikzcd}
    R(P)\otimes_AR(Y)\ar[r,"\widetilde{\varphi}_{P,Y}"]\ar[d,swap,"R(p)\otimes_A\text{id}"]
    &R(P\ot Y)\ar[d,"R(p\otimes\text{id})"]\\
    R(X)\otimes_AR(Y)\ar[r,"\widetilde{\varphi}_{X,Y}",swap]&R(X\ot Y)
    \end{tikzcd}    
    \end{center}
    Now from (iii) the top arrow $\widetilde{\varphi}_{P,Y}$ is an isomorphism, and since $p$ is epic and $R$ and $\otimes$ are exact, then $R(p\otimes\text{id})$ is epic and thus $\widetilde{\varphi}_{X,Y}$ has to be epic as well. An analogous argument using an injective hull of $X$, shows that $\widetilde{\varphi}_{X,Y}$ is also mono.
\end{enumerate}

\begin{Lemma}\label{lemma_algebra}
Let $\cc$ be a $G$-graded finite tensor category, then there exists a commutative algebra $A$ in $\cz_G(\cc)_e=\cz(\cc_e)$ and an equivalence $\cc\simeq\mod_A(\cz_G(\cc))$ of $G$-graded tensor categories.
\begin{proof}
Follows considering the construction given above.
\end{proof}
\end{Lemma}
\begin{Proposition}
Let $\cc$ and $\cd$ be $G$-graded finite tensor categories. Provided that $\cz_G(\cc)$ and $\cz_G(\cd)$ are equivalent as central $G$-extensions, then $\cc$ and $\cd$ are graded Morita equivalent.
\begin{proof}
Let $B$ be the commutative algebra in $\cz_G(\cd)_e$ constructed in Lemma \ref{lemma_algebra} and let $\Lambda:\cz_G(\cd)\simeq\cz_G(\cc)$ be an equivalence of central $G$-extensions, then $L:=\Lambda(B)$ is a commutative algebra in $\cz_G(\cc)_e$ and 
\begin{equation}
    \label{D=mod_L}
    \cd\simeq\mod_B(\cz_G(\cd))\simeq\mod_L(\cz_G(\cc))
\end{equation}
as $G$-graded tensor categories, where the first equivalence comes from Lemma \ref{lemma_algebra} and the second equivalence is induced by $\Lambda$.\\

Now notice that $\mod_{F(L)}(\cc)$ is an exact $G$-graded $\cc$-module category: let $A$ be the commutative algebra in $\cz_G(\cc)_e$ from Lemma \ref{lemma_algebra}, then $\cc\simeq\mod_A(\cz_G(\cc))$.
\begin{enumerate}[(i)]
    \item Since $\cd\simeq\mod_B(\cz_G(\cd))$ is exact over $\cz_G(\cd)$, then $\mod_L(\cz_G(\cc))$ is exact as a $\cz_G(\cc)$-module category.
    \item From \cite[Proposition 7.12.14]{EGNO} and Proposition \ref{grading_module_functors}, the category
    \begin{align*}
        \Fun_{\cz_G(\cc)}(\mod_L(\cz_G(\cc)),\cc)\simeq{}_L\bimod_A(\cz_G(\cc))
    \end{align*}
    is an exact $G\op$-graded module category over $(\cz_G(\cc))_\cc^{*}\simeq{}_A\bimod_A(\cz_G(\cc))\op$.
    \item Since $\cc\simeq \mod_A(\cz_G(\cc))$, then $\mod_{F(L)}(\cc)\simeq{}_L\bimod_A(\cz_G(\cc))$.
    \item From (ii) and (iii) it follows that $\mod_{F(L)}(\cc)$ is an exact  $G\op$-graded module category over
    $$(\cz_G(\cc))_\cc^{*}\simeq((\cc_e\boxtimes\cc\op)_\cc^{*})_\cc^{*}\simeq \cc_e\boxtimes\cc\op$$
    therefore it is in particular exact over $\cc_e$, and thus $\cm_e=\mod_{F(L)}(\cc_e)$ is an exact $\cc_e$-module category. From Lemma \ref{exact_induction} it follows that $\mod_{F(L)}(\cc)$ is exact over $\cc$.
\end{enumerate}

On the other hand, the module category $\cd\simeq\mod_B(\cz_G(\cd))$ is indecomposable over the dual category $(\cd_e\boxt\cd\op)_\cd^{*}\simeq\cz_G(\cd)$, hence under $\Lambda$ the module category $\mod_L(\cz_G(\cc))$ is indecomposable over $\cz_G(\cc)$. But the forgetful image $\mod_{F(L)}(\cc)$ might be decomposable over $\cc$. Since $\mod_{F(L)}(\cc)$ is exact over $\cc$, it follows from Proposition \ref{algebra_decomposition} that there is a decomposition of $\cc$-module categories
\begin{align*}
    \mod_{F(L)}(\cc)\simeq\bigoplus_{i\in I}\mod_{L_i}(\cc)
\end{align*}
where $L_i$ is an exact indecomposable algebra in $\cc_e$ for all $i\in I$ and $F(L)$ decomposes as $\prod_{i\in I} L_i$ as an algebra. Furthermore, the category of bimodules decomposes as follows
\begin{align}\label{bimod_decomp}
    {}_{F(L)}\bimod_{F(L)}(\cc)\simeq\bigoplus_{i,j\in I} {}_{L_i}\bimod_{L_j}(\cc)
\end{align}
Now consider the following commutative diagram,
\begin{equation*}
\begin{tikzcd}
\cz_G(\cc)\ar[dd,"Z\mapsto Z\ot L",swap]\ar[rr,"S_G"]&&\cz_G\left({}_{L_i}\bimod_{L_i}(\cc)\right)\ar[dd,"F_i"]\\\\
\mod_{L}(\cz_G(\cc))\subset{}_{L}\bimod_{L}(\cz_G(\cc))\ar[r,"\overline{F}",swap]
&{}_{F(L)}\bimod_{F(L)}(\cc) \ar[r,"\pi_i",swap]&{}_{L_i}\bimod_{L_i}(\cc)
\end{tikzcd}
\end{equation*}
where the equivalence $S_G$ comes from Proposition \ref{equivariant_centers_eq} and $\pi_i$ is the canonical projection of the direct sum (\ref{bimod_decomp}). Now since the forgetful functor $F_i$ is surjective (see Remark \ref{C_exact_over_center} (i)), then we have a surjective graded functor
\begin{equation*}
    H_i:=\pi_i\,\circ\,\overline{F}:\;\mod_{L}(\cz_G(\cc))\longrightarrow{}_{L_i}\bimod_{L_i}(\cc) 
\end{equation*}
between tensor categories of the same Frobenius-Perron dimension, and thus $H_i$ is an equivalence. Indeed, 
\begin{enumerate}[(i)]
    \item From \cite[Corollary 7.16.7]{EGNO} $\FP({}_{L_i}\bimod_{L_i}(\cc))=\FP(\cc)$.
    \item Since $\Lambda$ is an equivalence, $$\frac{1}{|G|}\FP(\cd)^2=\FP(\cz_G(\cd))=\FP(\cz_G(\cc))=\frac{1}{|G|}\FP(\cc)^2$$
    and thus $\FP(\cd)=\FP(\cc)$.
    \item From  (ii) and equivalence (\ref{D=mod_L}) it follows that
    $$\FP(\mod_{L}(\cz_G(\cc)))=\FP(\cd)=\FP(\cc)$$
\end{enumerate}
Summarizing, there is an exact $G$-graded $\cc$-module category $\cm=\mod_{L_i}(\cc)$ and a graded tensor equivalence $\cd\simeq{}_{L_i}\bimod_{L_i}(\cc)\simeq\cc_\cm^{*}\op$.
\end{proof}
\label{eq_centers_imply_Morita}
\end{Proposition}

Considering the results of Proposition \ref{eq_centers_imply_Morita} and Theorem \ref{equivalent_equivariant_centers}, we obtain the following theorem.

\begin{Theorem}\label{charac_graded_Morita_eq}
Two $G$-graded finite tensor categories $\cc$ and $\cd$ are graded Morita equivalent if and only if $\cz_G(\cc)$ and $\cz_G(\cd)$ are equivalent as braided $G$-crossed tensor categories.
\end{Theorem}

Taking into account \cite[Proposition 8.11]{DN}, Theorem \ref{charac_graded_Morita_eq} can be also described considering the equivariant centers as central $G$-extensions.

%\bibliographystyle{alpha}
%\bibliography{biblio}

\end{document}